\documentclass[10pt]{article}
\usepackage[T1]{fontenc}
\usepackage{lmodern}
\usepackage[english]{babel} %

\usepackage{amssymb,amsmath,amsthm,enumitem,hyperref}

\theoremstyle{plain}
\newtheorem{theorem}{Theorem}    % reset theorem numbering
\newtheorem*{theorem*}{Theorem}  % reset theorem not numbering

\newtheorem{lemma}[theorem]{Lemma}

\theoremstyle{definition}
\newtheorem{remark}[theorem]{Remark}

\setlist[enumerate]{%
  labelwidth=\parindent,
  labelindent=\parindent,
  leftmargin=0pt,
  labelsep=*,
  align=left,
  itemindent=\dimexpr\parindent+1.5em\relax,
  itemsep=2pt,
  topsep=2pt,
  partopsep=0pt,
  parsep=0pt
}

\title{A Note on Lagrange Subsets of Finite Groups}
\author{Mikhail Kabenyuk}
\date{}

\begin{document}
\maketitle
\begin{abstract}
In a finite group, a subset is called a Lagrange subset if its size divides the group order, and a factor if it admits a complementary subset. We provide a new and comparatively direct proof of the classification of groups in which every Lagrange subset is a factor. We show that any nontrivial such group must be a cyclic group of prime order, the cyclic group of order 4, or an elementary abelian group of order 4, 8, or 9.
\end{abstract}

\section{Introduction}
\label{section:Introduction}
Let $G$ be a finite group.  A subset $A \subseteq G$ is called a \textit{Lagrange subset} if $\lvert A\rvert$ divides $\lvert G\rvert$.  If there exists a subset $B \subseteq G$ such that every element $x\in G$ can be written \textit{uniquely} in the form
\[
x = ab,
\quad a\in A,\; b\in B,
\]
then $A$ is said to be a (left) \textit{factor} of $G$.  In this situation one immediately has
\[
\lvert G\rvert = \lvert A\rvert \,\lvert B\rvert,
\]
and both $A$ and $B$ are Lagrange subsets of $G$.

The problem of factoring a finite group $G$ into two (or more) Lagrange subsets goes back to Haj\'os’s 
combinatorial proof of Minkowski’s conjecture on lattice tilings by unit cubes \cite{Hajos1941}.  
In particular, Haj\'os\ showed that any lattice tiling of $\mathbb{R}^n$ by unit $n$-cubes must contain two 
cubes sharing a full $(n-1)$-face, and he recast that argument in group‑theoretic terms to characterize exactly 
which cyclic groups admit factorizations by arithmetic progressions \cite{Hajos1941}.  
In the broader abelian setting, subsequent work of R\'edei \cite{Redei1952}, 
Sands \cite{Sands1974}, Szab\'o \cite{Szabo2004}, and others has connected these set‑factorizations with tiling problems, zero‑sum sequences, and 
related combinatorial phenomena.  A good survey of these developments can be found in 
\cite{Szabo-Sands2009}.

Bernstein \cite{Bernstein1968} extended Haj\'os’s theorem to the non‑abelian setting, showing that many of the 
combinatorial tiling arguments carry over to arbitrary finite groups.  A related thread concerns the 
existence of \textit{minimal logarithmic signatures} (MLS): conjecturally, every finite group $G$ admits 
a factorization
\[
G = A_1 A_2 \cdots A_k
\]
where each $\lvert A_i\rvert$ is prime or equals $4$.  
It is known that every finite solvable group has an 
MLS, and that if a normal subgroup $K\lhd G$ and the quotient $G/K$ both admit MLS then so does $G$ 
\cite{Rahimipour2018}.  Consequently, any counterexample to the MLS conjecture must be a non-abelian 
simple group, reducing the problem to the finite simple groups.

Hooshmand \cite[Problem 19.35]{Kourovka2018} 
(see also \cite{Hooshmand2014,Banakh2018multifactorizable, Hooshmand2021}) 
asked whether for every factorization $\lvert G\rvert = m_1\cdots m_k$, there exist subsets 
$A_i\subseteq G$ with $\lvert A_i\rvert=m_i$ and $G=A_1\cdots A_k$.  
The case $k=2$ has been studied in \cite{Bildanov2020, Hooshmand2021}, where it is shown that a finite 
group is $2$‑factorizable if and only if every finite simple group is $2$‑factorizable, 
and all simple groups of order up to $10~000$ are $2$-factorizable \cite{Bildanov2020}.  

For $k=3$, Bergman \cite{Bergman2020} proved that the alternating group $A_4$ is not $3$‑factorizable.
Brunault \cite{Banakh2018Factorizable} used computer methods to show that $A_5$ has no 
$(2,3,5,2)$\nobreakdash-factorization (in fact, $A_5$ does not even admit 
a $(2,15,2)$\nobreakdash-factorization; this was shown in \cite{Kab2021} without computer calculations). 
More generally, it is proved in \cite{Kab2021} that for every integer $k\geq3$, 
there exist integers $m_1,m_2,\ldots,m_k$ all greater than $1$
and a finite group of order $m_1m_2\ldots m_k$ that does not have any 
$(m_1,m_2,\ldots,m_k)$\nobreakdash-factorization.
Among groups of order at most $100$, there are only eight groups that do not admit a $k$-factorization 
for at least one $k\geq2$; all other groups are $k$-factorizable for every admissible $k$ \cite{Kab2021}.
It was also shown that the simple groups of orders $168$ and $360$ are $k$‑factorizable 
for every admissible $k$ \cite{Kab2024}.

Recently, Chin, Wang, and Wong \cite{Chin-Wang-Wong} introduced the notion of \textit{complete 
factorization}. Subsets $A_1, \ldots, A_k$ of $G$ form a complete factorization if they are pairwise 
disjoint and every $g \in G$ is uniquely represented as $g = a_1 \cdots a_k$ with $a_i \in A_i$. 
They proved that if $G$ is a finite abelian group and $|G|=m_1\ldots m_k$ 
where $m_1,\ldots,m_k$ are integers greater than $1$ and $k>2$, then 
there exist subsets $A_1,\ldots,A_k$ of $G$ (with $|A_i|=m_i$ for all $i=1,2,\ldots,k$) 
which form a complete factorization of the group $G$.
This result was extended to nilpotent groups in \cite{Kab2023}.

A finite group $G$ is said to have the \textit{strong CFS} (converse of factorization by subsets) 
property if every Lagrange subset of $G$ is a factor of $G$. The notions of a Lagrange subset and a 
strong CFS group were introduced in \cite{HooshKoh2025}, where the following was proved:

\begin{theorem}[\cite{HooshKoh2025}]
A nontrivial finite group $G$ satisfies the strong CFS property if and only if $G$ is either
\begin{itemize} 
  \item a cyclic group of prime order, or
  \item one of the four groups $C_4$, $C_2^2$, $C_2^3$, or $C_3^2$.
\end{itemize}
\end{theorem}

The authors of \cite{HooshKoh2025} used Cayley graph arguments to constrain the possible orders of $G$, 
reducing the classification to an examination of about twenty concrete groups. In this paper, we offer a 
streamlined proof of the "only if" direction, which breaks down into four independent steps:
\begin{enumerate}[label=(\arabic*)]
  \item Analysis of groups with "large" proper subgroups (Lemma \ref{lemma:main});
  \item Groups of even order containing elements of odd order (Lemma \ref{lemma:small subset});
  \item Investigation of groups of order $8$ (Lemma \ref{lemma:groups of order 8});
  \item A direct check for the three remaining groups: $D_4$, $C_8$, and $C_9$.
\end{enumerate}

\vspace{1ex}
\noindent\textbf{Notation.} Throughout, $G$ is a finite group with identity $e$. For any $A \subseteq G$ and $x \in G$, we denote $Ax = \{ax \mid a \in A\}$. We write $\langle g \rangle$ for the cyclic subgroup generated by $g$. A subset $A$ is a \textit{left factor} if $G = \bigsqcup_{b \in B} Ab$ for some $B \subseteq G$. 

With these conventions in place, we turn in Section \ref{section:lemmas} to our lemmas, 
and then in Section \ref{section:proof theorem} to the completion of the proof.

\section{Lemmas}
\label{section:lemmas}
Our main objective is to provide a new proof of the “only if” direction of the theorem. 
While this approach dispenses with the geometric Cayley graph arguments (arguably a drawback), 
it avoids the tedious case-by-case verification of numerous concrete groups (a significant advantage). 
Our proof relies on two key lemmas (Lemma \ref{lemma:main} and Lemma \ref{lemma:groups of order 8}), which reduce the argument to the consideration of only three groups: $D_4$, $C_8$, and $C_9$.

\begin{lemma}\label{lemma:subset of subgroup}
Let $H$ be a subgroup of $G$ and let $A\subseteq H$.
If $A$ is a factor of $H$, then $A$ is a factor of $G$.
\end{lemma}

\begin{proof}
Since $A$ is a factor of $H$, there exists a subset $B'\subseteq H$ such that
every element of $H$ can be written uniquely in the form $ab'$ with
$a\in A$ and $b'\in B'$. In particular, $H = A B'$.

Let $B''$ be a set of representatives for the left cosets of $H$ in $G$.
Then $G = H B''$. Set
\[
B = B' B'' \subseteq G.
\]
We have
\[
G = H B'' = A B' B'' = A B,
\]
so every element of $G$ can be written in the form $ab$ with $a\in A$ and $b\in B$.
The uniqueness of this representation follows directly from the uniqueness of the coset decomposition.
Thus every element of $G$ has a unique representation as $ab$ with $a\in A$ and $b\in B$,
and $A$ is a factor of $G$.
\end{proof}

\begin{lemma}
\label{lemma:small subset}
\begin{enumerate}[label=\textup{(\roman*)}]\!
    \item Let $G$ be a group of even order, and let $a\in G$, $a\neq e$. 
    The subset $\{e,a\}$ is a factor of $G$ if and only if the order of $a$ is even.
    \item Let $G$ be an elementary abelian $2$-group.
    Every subset of size $4$ is a factor of $G$.
    \item\,Let $G$ be an elementary abelian $3$-group.
    Every subset of size $3$ is a factor of $G$.
\end{enumerate}
\end{lemma}
\begin{proof}[Proof of \textup{(i)}]
First suppose that the order of $a$ is even, say $|a| = 2m$ for some positive integer $m$.
Let $H = \langle a\rangle$.
Set
\[
B = \{e,a^2,\ldots,a^{2m-2}\}.
\]
Clearly $H = \{e,a\} B$, and it is straightforward to check that
$B \cap aB = \varnothing$.
Thus $\{e,a\}$ is a factor of $H$.
By Lemma~\ref{lemma:subset of subgroup}, $\{e,a\}$ is then a factor of $G$.

Conversely, suppose that $\{e,a\}$ is a factor of $G$. 
Then there exists a subset $B\subseteq G$ such that 
\[
G = B \sqcup aB
\]
is a disjoint union.
Let $H = \langle a\rangle$. We have
\[
H = (B \cap H)\ \sqcup\ (aB \cap H).
\]
Since left multiplication by $a$ induces a bijection from $B$ onto $aB$ and
the set $H$ is invariant under left multiplication by $a$, it follows that
left multiplication by $a$ gives a bijection
\[
B \cap H \longrightarrow aB \cap H,\qquad b \mapsto ab.
\]
In particular,
\[
|B \cap H| = |aB \cap H|.
\]
Therefore,
\[
|H|
  = |B \cap H| + |aB \cap H|
  = 2\,|B \cap H|,
\]
which is even. Since $|H| = |a|$, the order of $a$ is even.
This proves (i).
\end{proof}

\begin{proof}[Proof of \textup{(ii)}]
Let $A\subseteq G$ be a subset of size $4$.
Since the existence of a factorization is invariant under left translation of $A$, 
we may assume that $e\in A$.
Write
\[
A = \{e,x,y,z\}.
\]

Let $H = \langle A\rangle$. Then $H$ is an elementary abelian $2$-group
generated by at most three elements, hence $|H|$ is either $4$ or $8$.
By Lemma~\ref{lemma:subset of subgroup}, it suffices to show that $A$ is a factor of $H$.

If $|H| = 4$, then $A = H$, and $A$ is a factor of $H$ with complement $\{e\}$.

Now assume that $|H| = 8$. There are two possibilities.
If $xyz = e$, then $z = xy$ and $A$
is a subgroup of $H$, so $A$ is a factor of $H$.

If $xyz \neq e$, then
set $t = xyz$ and $B = \{e,t\}$. We have $AB=H$ and $A\cap At=\varnothing$, so
$A$ is a factor of $H$.

In both cases $A$ is a factor of $H$, and therefore a factor of $G$.
This proves~(ii).
\end{proof}

\begin{proof}[Proof of \textup{(iii)}]
Let $A\subseteq G$ be a subset of size $3$.
As in the proof of (ii), we may assume that $e\in A$.
Write
\[
A=\{e,x,y\}.
\]

Let $H=\langle A\rangle$. Then $H$ is an elementary abelian $3$-group
generated by at most two elements, hence $|H|$ is either $3$ or $9$.
By Lemma~\ref{lemma:subset of subgroup}, it suffices to show that $A$ is
a factor of $H$.

If $|H|=3$, then $A=H$, and $A$ is a factor of $H$ with complement $\{e\}$.

Now assume that $|H|=9$.
Set $t = xy$ and $B = \{e, t, t^2\}$. 
A direct check shows that $AB = H$ and that the three sets $A$, $At$, and $At^2$ 
are pairwise disjoint, so $A$ is a factor of $H$.
This proves~(iii).
\end{proof}

\begin{remark} 
\label{rem:translation}
The property of a subset $A \subseteq G$ being a factor is invariant under translation. 
That is, if $G = AB$ is a factorization, then for any $u,v \in G$, the sets $uA$ and $Bv$ 
also form a factorization $G = (uA)(Bv)$.
Consequently, throughout the proofs, we may assume without loss of generality that the identity $e$ lies 
in the complement $B$ (by replacing $B$ with $Bb^{-1}$ for some $b \in B$).
In particular, this implies that for every $x \in B \setminus \{e\}$, 
the translate $Ax$ is disjoint from $A$. 
We will use this observation implicitly in what follows.
\end{remark}

\begin{lemma}
\label{lemma:main}
Let $H$ be a proper subgroup of a finite group $G$ with $\lvert H\rvert \ge 5$.
Choose $h_0\in H$ with $h_0\neq e$ and choose $g\in G\setminus H$.
Let $h_1$ be an element of $H$ not in the set
$\{e,h_0,h_0^{-1}\}$,
and set
\[
H'= H\setminus\{e,h_0\}.
\]
Define
\begin{equation}\label{equ:subset-not-factor}
A= H'\cup\{g,h_1g\}.
\end{equation}
Then $\lvert A\rvert = \lvert H\rvert$ and $A$ is not a left factor of $G$.
\end{lemma}

\begin{proof}
Suppose, for contradiction, that $A$ is a left factor of $G$, i.e., 
$$
G = \bigsqcup_{b\in B} Ab.
$$
By Remark~\ref{rem:translation}, we assume $e \in B$.
Since by construction $e \notin A$, the identity element must be covered 
by some translate $Ab$ with $b \in B \setminus \{e\}$.
Thus, there exist $a \in A$ and $b \in B \setminus \{e\}$ such that $ab = e$ (so $b = a^{-1}$), 
and the condition of factorization implies
\[
A \cap Ab = \varnothing.
\]
We examine the three possibilities for the element $a \in A$:
\[
a\in H',\quad a = g,\quad\text{or}\quad a = h_1g.
\]

\begin{enumerate}
  \item[\textbf{Case}] $a\in H'$. 
    Since  $b=a^{-1}\in H$ and $A\cap Ab=\varnothing$, we have
    $     
      H'b\cup H' \subseteq H,
    $
    with $H'b\cap H'=\varnothing$. 
   Comparing the cardinalities gives
    \[
      \lvert H\rvert\ge\lvert H'b\cup H'\rvert
      = \lvert H'b\rvert + \lvert H'\rvert
      = 2(\lvert H\rvert -2),
    \]
    so $\lvert H\rvert\leq4$, contradicting $\lvert H\rvert\geq5$.
      
  \item[\textbf{Case}] $a = g$.  Then $b=g^{-1}\notin H$ and $\{g,h_1g\}b=\{e,h_1\}$.  
  But $h_1\in H'\subset A$, so $h_1\in A\cap A b$, a contradiction.

  \item[\textbf{Case}] $a = h_1g$.  Then $b = g^{-1}h_1^{-1}\notin H$    
  and $\{g,h_1g\}b = \{e,h_1^{-1}\}$.  
  Since $h_1^{-1}\in H'\subset A$, we get $h_1^{-1}\in A\cap A b$, again a contradiction.    
\end{enumerate}

Having ruled out all cases, no such $B$ can exist, and $A$ is not a left factor of $G$.
\end{proof}

\begin{lemma}
\label{lemma:groups of order 8}
Let $G$ be a group of order $8$, and let $a,b\in G$ be two elements of order $4$ with $\langle a\rangle\neq\langle b\rangle$.
Then
\[
  A = \{a,a^2,a^3,b\}
\]
is not a left factor of $G$.
\end{lemma}
\begin{proof}
    Suppose to the contrary that there is an element $x\in G$ such that
    \[
    G=A\cup Ax \text{ and } A\cap Ax=\varnothing.
    \]
    Then $e\in Ax$ and hence $x=a^{-i}$ for some $i\in\{1,2,3\}$, or $x=b^{-1}$.
    If $x=a^{-i}$, then $a\in A\cap A x$, contradicting disjointness.
    If $x=b^{-1}=b^3$, then
    \[
    Ax=\{ab^3,a^2b^3,a^3b^3,e\}.
    \]
    Since $b^3\notin A$ and also $b^3\notin A x$,
    the element $b^3$ is not contained in $A\cup A x=G$, a clear contradiction.
\end{proof}

\section{Proof of the Theorem}
\label{section:proof theorem}
As already noted, the “if” direction follows directly from Lemma \ref{lemma:small subset}. 
We are now ready to establish the “only if” part of the theorem.

\begin{proof}
Let $G$ be a finite group in which every Lagrange subset is a factor. 
We may assume that $G$ is not a cyclic group of prime order.

First, suppose that $|G|$ is even. By Lemma~\ref{lemma:small subset}\,(i), 
every non-identity element of $G$ must have even order. 
It follows that $G$ is a $2$-group.
Since $G$ cannot have a proper subgroup of order greater than $4$ (by Lemma~\ref{lemma:main}), 
the order of $G$ must be $4$ or $8$.
If $|G|=4$, $G$ is either $C_4$ or $C_2^2$, both of which appear in the theorem's list.
If $|G|=8$, Lemma~\ref{lemma:groups of order 8} implies that $G$ cannot contain 
two distinct cyclic subgroups of order $4$. 
This eliminates the quaternion group $Q_8$ and the product $C_4 \times C_2$. 
Since $C_2^3$ is a solution, the only remaining candidates to check are $D_4$ and $C_8$.

Next, suppose that $|G|$ is odd. Then Lemma~\ref{lemma:main} implies that the only prime dividing $|G|$ 
is $3$. 
Furthermore, $|G|$ cannot be $27$ or larger, as such a group would contain a subgroup of order $9$. 
Thus, the only possible odd composite order is $9$. 
The elementary abelian group $C_3^2$ is a solution, so we must only check $C_9$.

It remains to show that each of the three groups 
\[
D_4,\ C_8,\ C_9
\]
admits a Lagrange subset which is not a left factor.

\vspace{1ex}
\noindent\textbf{Case} $G = D_4$.
Let
\[
G = \langle a,b \mid a^4 = b^2 = e,\; b a b = a^{-1}\rangle.
\]
Consider the Lagrange subset
\[
A = \{a,a^2,b,a^2b\}\subset G.
\]
If $A$ were a left factor, then there would be some $x\in G$ with
$e\in Ax$, and hence $x\in\{a^{-1},a^2,b,a^2b\}$.
Using that $a^2$ is central in $G$, we check:
\[
\begin{aligned}
Aa^{-1}   &= \{e,a,ba^{-1},ba\},\\
Aa^2      &= \{a^3,e,ba^2,b\},  \\
Ab        &= \{ab,a^2b,e,a^2\}, \\
Aa^2b    &= \{a^3b,b,a^2,e\}.
\end{aligned}
\]
In all cases, $A \cap Ax \neq \varnothing$, contradicting the disjointness required for a factorization.  Therefore $A$ is not a left factor of $D_4$.

\vspace{1ex}
\noindent\textbf{Case} $G = C_8$.
Let
\[
G = \langle a \mid a^8 = e\rangle.
\]
Consider the Lagrange subset
\[
A = \{a,a^2,a^3,a^5\}.
\]
If $A$ were a left factor, there would be some $x\in G$ with $e\in A x$, forcing
\[
x\in \{a^{-1},a^{-2},a^{-3},a^{-5}\}.
\]
A direct check shows that for each such $x$, the right translate
$Ax$ meets $A$ non‑trivially. Hence $A$ cannot be a left factor of $C_8$.

\vspace{1ex}
\noindent\textbf{Case} $G = C_9$.
Let
\[
G = \langle a \mid a^9 = e\rangle,
\]
and consider the set
\[
A = \{a,a^2,a^4\}.
\]
Since $\lvert A\rvert = 3$ divides $\lvert G\rvert = 9$, $A$ is a Lagrange subset.
If $A$ were a left factor, one of its right translates would have to contain the identity.  Checking the candidates
\[
x \in\{a^{-1},a^{-2},a^{-4}\},
\]
one finds that only $x=a^{-4}=a^5$ yields a translate disjoint from $A$:
\[
Aa^5 = \{e,a^6,a^7\}.
\]
Thus any factorization would use the two translates $A$ and $Aa^5$, 
leaving the remaining set to be
\[
G \setminus \bigl(A \cup A a^5\bigr) = \{a^3,a^5,a^8\}.
\]
But no right translate of $A$, which always has the form $\{ax,a^2x,a^4x\}$,
can equal this set.
Hence $A$ is not a left factor of $C_9$.

Thus we conclude that any finite group not isomorphic to one of the groups
listed in the theorem fails to have the strong CFS property, completing the proof.
\end{proof}

\end{document}